\documentclass[11pt]{article}

\usepackage{graphicx}
\usepackage{latexsym}
\usepackage{amsmath}
\usepackage{amssymb}
\usepackage{natbib}
\usepackage{url}
\usepackage{amsthm}

\newtheorem{theorem}{Theorem}
\newtheorem{corollary}{Corollary}
\newtheorem{proposition}{Proposition}
\newtheorem{remark}{Remark}
\newtheorem{condition}{Condition}
\newtheorem{lemma}{Lemma}[section]

\newcommand{\bB}{\mbox{\boldmath {$B$}}}

\newcommand{\be}{\mbox{\boldmath {$e$}}}

\newcommand{\bh}{\mbox{\boldmath {$h$}}}
\newcommand{\bH}{\mbox{\boldmath {$H$}}}

\newcommand{\bI}{\mbox{\boldmath {$I$}}}

\newcommand{\bO}{\mbox{\boldmath {$O$}}}

\newcommand{\bP}{\mbox{\boldmath {$P$}}}

\newcommand{\br}{\mbox{\boldmath {$r$}}}

\newcommand{\bS}{\mbox{\boldmath {$S$}}}

\newcommand{\bu}{\mbox{\boldmath {$u$}}}

\newcommand{\bV}{\mbox{\boldmath {$V$}}}

\newcommand{\bx}{\mbox{\boldmath {$x$}}}
\newcommand{\bX}{\mbox{\boldmath {$X$}}}
\newcommand{\by}{\mbox{\boldmath {$y$}}}

\newcommand{\bz}{\mbox{\boldmath {$z$}}}
\newcommand{\bZ}{\mbox{\boldmath {$Z$}}}
\newcommand{\bze}{\mbox{\boldmath {$0$}}}
\newcommand{\bone}{\mbox{\boldmath {$1$}}}

\newcommand{\bmu}{\mbox{\boldmath $ \mu $}}

\newcommand{\bSigma}{\mbox{\boldmath $ \Sigma $}}

\newcommand{\bLam}{\mbox{\boldmath $ \Lambda $}}

\newcommand{\bnu}{\mbox{\boldmath $ \nu $}}

\newcommand{\tr}{\mbox{tr}}

\newcommand{\var}{\mbox{var}}

\newcommand{\plim}{\mathop{\rm plim}\limits}

\long\def\symbolfootnote[#1]#2{\begingroup%
\def\thefootnote{\fnsymbol{footnote}}\footnote[#1]{#2}\endgroup}

\begin{document}

\begin{center}
\Large
{\bf Principal component analysis based clustering for high-dimension, low-sample-size data}
\end{center}
\begin{center}
\vskip 0.5cm
\textbf{\large Kazuyoshi Yata and Makoto Aoshima} \\
Institute of Mathematics, University of Tsukuba, Ibaraki, Japan \\[-1cm]
\end{center}
\symbolfootnote[0]{\normalsize Address correspondence to Makoto Aoshima, 
Institute of Mathematics, University of Tsukuba, Ibaraki 305-8571, Japan; 
Fax: +81-298-53-6501; E-mail: aoshima@math.tsukuba.ac.jp}

\begin{abstract}
In this paper, we consider clustering based on principal component analysis (PCA) for high-dimension, low-sample-size (HDLSS) data. 
We give theoretical reasons why PCA is effective for clustering HDLSS data. 
First, we derive a geometric representation of HDLSS data taken from a two-class mixture model. 
With the help of the geometric representation, we give geometric consistency properties of sample principal component scores in the HDLSS context.
We develop ideas of the geometric representation and geometric consistency properties to multiclass mixture models.
We show that PCA can classify HDLSS data under certain conditions in a surprisingly explicit way.
Finally, we demonstrate the performance of the clustering by using microarray data sets. 
\end{abstract}
\section{Introduction}
High-dimension, low-sample-size (HDLSS) data situations occur in many areas of modern science such as genetic microarrays, medical imaging, text recognition, finance, chemometrics, and so on. 
In recent years, substantial work has been done on HDLSS asymptotic theory, where the sample size $n$ is fixed or $n/d\to 0$ as the data dimension $d\to\infty$. 
\citet{Hall:2005}, \citet{Ahn:2007}, \citet{Yata:2012} and \citet{Lv:2013} explored several types of geometric representations of HDLSS data. 
\citet{Jung:2009} showed inconsistency properties of the sample eigenvalues and eigenvectors in the HDLSS context.
\citet{Yata:2012} developed the noise-reduction methodology to give consistent estimators of both the eigenvalues and eigenvectors together with principal component (PC) scores in the HDLSS context. 
\citet{Hellton:2014} also gave several asymptotic properties of the sample PC scores in the HDLSS context. 
On the other hand, the asymptotic behavior of the sample eigenvalues was studied by \citet{Johnstone:2001} and several literatures in high-dimension, large sample size data situations such as $n/d \to c>0$. 

The HDLSS asymptotic theory was created under the assumption either the population distribution is Gaussian or the random variables in a sphered data matrix have a $\rho$-mixing dependency. 
However, \citet{Yata:2010} developed a HDLSS asymptotic theory without such assumptions. 
Moreover, they created a new principal component analysis (PCA) called the cross-data-matrix methodology that is applicable to constructing an unbiased estimator in HDLSS nonparametric settings. 
Meanwhile, PCA is quite popular for clustering high dimensional data. 
See Section 9.2 in \citet{Jolliffe:2002} for details. 
For clustering HDLSS gene expression data, see \citet{Armstrong:2002} and \citet{Pomeroy:2002}. 
\citet{Liu:2008} and \citet{Ahn:2012} gave binary split type clustering methods for HDLSS data.
Given this background, we decided to focus on high-dimensional structures of multiclass mixture models.
In this paper, we consider asymptotic properties of PC scores for high-dimensional mixture models to apply to cluster analysis in HDLSS settings. 
The main contribution of this paper is that we give theoretical reasons why PCA is effective for clustering HDLSS data. 

Suppose there are independent and $d$-variate populations, $\Pi_i,\ i=1,...,k$, having an unknown mean vector $\bmu_{i}$ and unknown covariance matrix $\bSigma_{i}(\ge \bO)$ for each $i$. 
We do not assume $\bSigma_{1}=\cdots=\bSigma_{k}$.
The eigen-decomposition of $\bSigma_{i}$ is given by $\bSigma_{i}=\bH_{i}\bLam_{i}\bH_{i}^T$, where $\bLam_{i}=\mbox{diag}(\lambda_{i1},...,\lambda_{id})$ having eigenvalues $\lambda_{i1}\ge \cdots \ge \lambda_{id} \ge 0$ and $\bH_{i}$ is an orthogonal matrix of the corresponding eigenvectors. 
We consider a mixture model to classify a data set into $k\ (\ge 2)$ groups. 
We assume that any sample is taken with mixing proportions $\varepsilon_i$s from $\Pi_i$s, where $\varepsilon_i\in (0,1)$ and $\sum_{i=1}^k\varepsilon_i=1$ but {\it the label of the population is missing}.
We assume that $\varepsilon_i$s are independent of $d$. 
We consider a mixture model whose probability density function (or probability function) is given by
\begin{align}
f(\bx)=\sum_{i=1}^k \varepsilon_i \pi_i(\bx; \bmu_i,\bSigma_i),
\label{1.1}
\end{align}
where $\bx\in \mathbb{R}^d$ and $\pi_i(\bx; \bmu_i,\bSigma_i)$ is a $d$-dimensional probability density function (or probability function) of $\Pi_i$ having a mean vector $\bmu_i$ and covariance matrix $\bSigma_i$. 
Suppose we have a $d\times n$ data matrix $\bX=(\bx_{1},...,\bx_{n})$, where $\bx_{j},\ j= 1,...,n$, are independently taken from (\ref{1.1}). 
We assume $n\ge k$. 
Let $n_i=\# \{j| \bx_j\in \Pi_i\ \mbox{for $j=1,...,n$} \}$ and $\eta_i=n_i/n$ for $i=1,...,k$, where $\# A$ denotes the number of elements in a set $A$. 
We assume that $n$ and $n_i$s are independent of $d$. 
Let $\bmu$ and $\bSigma$ be the mean vector and the covariance matrix of (\ref{1.1}).
Then, we have that $\bmu=\sum_{i=1}^k \varepsilon_i \bmu_i$ and $\bSigma= \sum_{i=1}^{k-1}\sum_{j=i+1}^k \varepsilon_i \varepsilon_{j}(\bmu_{i}-\bmu_{j})(\bmu_{i}-\bmu_{j})^T+\sum_{i=1}^k \varepsilon_i\bSigma_i$.
We note that $E(\bx | \bx \in \Pi_i)=\bmu_i$ and $\var(\bx | \bx \in \Pi_i)=\bSigma_i$ for $i=1,...,k$. 
We denote the eigen-decomposition of $\bSigma$ by $\bSigma=\bH\bLam \bH^T$, where $\bLam=\mbox{diag}(\lambda_{1},...,\lambda_{d})$ having eigenvalues $\lambda_{1}\ge \cdots \ge \lambda_{d}\ge 0$ and $\bH=(\bh_{1},...,\bh_{d})$ is an orthogonal matrix of the corresponding eigenvectors. 
Let $\bx_j-\bmu=\bH\bLam^{1/2}(z_{1j},...,z_{dj})^T$ for $j=1,...,n$. 
Then, $(z_{1j},...,z_{dj})^T$ is a sphered data vector from a distribution with the identity covariance matrix. 
The $i$th true PC score of $\bx_j$ is given by $\bh_i^T(\bx_j-\bmu)=\lambda_i^{1/2} z_{ij}$ (hereafter called $s_{ij}$). 
We note that $\var(s_{ij})=\lambda_i$ for all $i,j$. 
Let $\bmu_{i,j}=\bmu_i-\bmu_j$ and $\Delta_{i,j}=||\bmu_{i,j}||^2$ for $i,j=1,...,k\ (i<j)$, where $||\cdot||$ denotes the Euclidean norm. 
Let $\Delta_{\min}=\min_{1\le i<j \le k}\Delta_{i,j}$. 
We note that $\Delta_{\min}=\Delta_{1,2}$ when $k=2$. 
Since the sign of an eigenvector is arbitrary, we assume that $\bh_i^T\bmu_{i,i+1}\ge 0$ for $i=1,...,k-1$, without loss of generality.
In addition, for the largest eigenvalue $\lambda_{i1}$s, we assume the following condition as necessary:
\begin{condition}
\label{con1}
$\displaystyle \frac{\max_{i=1,...,k}\lambda_{i1}}{\Delta_{\min}}\to 0$ \ as $d\to \infty$. 
\end{condition}
 
We consider clustering $\bx_1,...,\bx_n$ into one of $\Pi_i$s in HDLSS situations. 
When $k=2$, \citet{Yata:2010} gave the following result: 
We denote the angle between two vectors $\bx$ and $\by$ by $\mbox{Angle}(\bx,\by)=\cos^{-1}\{\bx^T\by/(||\bx||\cdot ||\by||)\}$. 
Under Condition 1, it holds that as $d\to \infty$ 
\begin{align}
\frac{\lambda_1}{\varepsilon_1\varepsilon_2 \Delta_{1,2}}\to 1\quad \mbox{and} \quad \mbox{Angle}(\bh_1,\bmu_{1,2})\to 0. 
\label{1.2}
\end{align}
Furthermore, for the normalized first PC score $s_{1j}/\lambda_1^{1/2}\ (=z_{1j})$, it follows that 
\begin{eqnarray}
\plim_{d\to \infty} \frac{s_{1j} }{ \lambda_1^{1/2}}= \left\{ \begin{array}{ll}
 \sqrt{\varepsilon_2/\varepsilon_1} & \mbox{ when } \bx_j\in \Pi_1, \\[1mm]
-\sqrt{\varepsilon_1/\varepsilon_2} & \mbox{ when } \bx_j\in \Pi_2 
\end{array} \right. 
\label{1.3}
\end{eqnarray}
for $j=1,...,n$. 
Here, `$\plim$' denotes the convergence in probability.
One would be able to classify $\bx_j$s into two groups if $s_{1j}$ is accurately estimated in HDLSS situations. 

In this paper, we consider asymptotic properties of sample PC scores for (\ref{1.1}) in the HDLSS context such as $d\to \infty$ while $n$ is fixed. 
In Section 2, we first derive a geometric representation of HDLSS data taken from the two-class mixture model.
With the help of the geometric representation, we give geometric consistency properties of sample PC scores in the HDLSS context. 
We show that PCA can classify HDLSS data under certain conditions in a surprisingly explicit way. 
In Section 3, we investigate asymptotic behaviors of true PC scores for the $k\ (\ge 3)$-class mixture model and provide geometric consistency properties of sample PC scores when $k\ge 3$. 
In Section 4, we demonstrate the performance of clustering based on sample PC scores by using microarray data sets.  
We show that the real HDLSS data sets hold the geometric consistency properties.
\section{PC scores for two-class mixture model}
\subsection{Preliminary}
The sample covariance matrix is given by 
$\bS=(n-1)^{-1}(\bX-\overline{\bX})(\bX-\overline{\bX})^T=(n-1)^{-1}\sum_{j=1}^n(\bx_j-\bar{\bx}_n)(\bx_j-\bar{\bx}_n)^T$, where $\bar{\bx}_n=n^{-1}\sum_{j=1}^n \bx_j$ and $\overline{\bX}=\bar{\bx}_n\bone_n^T$ with $\bone_n=(1,...,1)^T\in \mathbb{R}^n$. 
Then, we define the $n \times n$ dual sample covariance matrix by $\bS_{D}=(n-1)^{-1}(\bX-\overline{\bX})^T(\bX-\overline{\bX})$. 
We note that $\mbox{rank}(\bS_D)\le n-1$. 
Let $\hat{\lambda}_{1}\ge\cdots\ge \hat{\lambda}_{n-1}\ge 0$ be the eigenvalues of $\bS_{D}$. 
Then, we define the eigen-decomposition of $\bS_{D}$ by $\bS_{D}=\sum_{i=1}^{n-1}\hat{\lambda}_{i}\hat{\bu}_{i}\hat{\bu}_{i}^T $, where $\hat{\bu}_{i}=(\hat{u}_{i1},...,\hat{u}_{in})^T$ denotes a unit eigenvector corresponding to $\hat{\lambda}_{i}$. 
Since the sign of $\hat{\bu}_{i}$s is arbitrary, we assume $\hat{\bu}_{i}^T\bz_i \ge 0$ for all $i$ without loss of generality, where $\bz_i$ is defined by $\bz_{i}=(z_{i1},...,z_{in})^T$. 
Note that $\bS$ and $\bS_{D}$ share the non-zero eigenvalues. 
Let $\hat{z}_{ij}=\hat{u}_{ij}n^{1/2}$ for $i=1,...,n-1;\ j=1,...,n$. 
We note that $\hat{z}_{ij}$ is an estimate of $s_{ij}/\lambda_i^{1/2}\ (=z_{ij})$ for $i=1,...,n-1;\ j=1,...,n$  from the facts that $\hat{z}_{ij}=\{n/(n-1)\}^{1/2} \hat{\bh}_i^T(\bx_j-\bar{\bx}_n)/\hat{\lambda}_i^{1/2}$ and $\sum_{j=1}^n \hat{z}_{ij}^2/n=1$ if $\hat{\lambda}_i>0$, where $\hat{\bh}_i$ denotes a unit eigenvector of $\bS$ corresponding to $\hat{\lambda}_{i}$. 
Let $\bX_0=\bX-\bmu \bone_n^T$ and $\bP_n=\bI_n-n^{-1}\bone_n \bone_n^T $, where $\bI_{n}$ denotes the $n$-square identity matrix. 
We note that $\bS_D=\bP_n\bX_0^T \bX_0 \bP_n/(n-1)$. 
We consider the sphericity condition:
$\tr(\bSigma^2)/\tr(\bSigma)^2\to 0$ as $d\to \infty$. 
When one can assume that $\bX$ is Gaussian or $\bZ=(z_{ij})$ is $\rho$-mixing, \citet{Ahn:2007} and \citet{Jung:2009} gave a geometric representation as follows: 
\begin{align}
\plim_{d\to \infty} \frac{\bX_0^T \bX_0}{\tr(\bSigma)}=\bI_n,\quad \mbox{so that}\quad 
\plim_{d\to \infty}
\frac{(n-1)\bS_{D}}{\tr(\bSigma)}=\bP_n.
\label{2.1}
\end{align}
\begin{remark}
\citet{Yata:2012} showed that (\ref{2.1}) holds under the sphericity condition and $ \var(||\bx_j-\bmu||^2)/\tr(\bSigma)^2\to 0$ as $d\to \infty$.
\end{remark}

From (\ref{2.1}), we observe that the eigenvalue becomes deterministic as the dimension grows while the eigenvector of $\bS_D$ does not uniquely determine the direction. 
We note that (\ref{1.1}) does not satisfy the assumption that $\bX$ is Gaussian or $\bZ$ is $\rho$-mixing. 
See Section 4.1.1 in \citet{Qiao:2010} for details.
\subsection{Geometric representation and consistency property of PC scores when $k=2$}
We will find a geometric representation for (\ref{1.1}) and the finding is completely different from (\ref{2.1}). 
We assume the following conditions:
\begin{condition}
\label{con2}
$\displaystyle \frac{\max_{i=1,...,k}\tr(\bSigma_i^2)}{ \Delta_{\min}^2}\to 0$ \ as $d\to \infty$.
\end{condition}
\begin{condition}
\label{con3}
$\displaystyle \frac{\max_{i=1,...,k}\var(||\bx-\bmu_i||^2|\bx \in \Pi_i )}{\Delta_{\min}^2}\to 0 $ \ as $d\to \infty$. 
\end{condition}
\begin{condition}
\label{con4}
$\displaystyle \frac{\tr(\bSigma_i)-\tr(\bSigma_j)}{\Delta_{\min}}\to 0 $ \ as $d\to \infty$ for all $i,j=1,...,k\ (i<j)$.
\end{condition}
\begin{remark}
If $\Pi_i$s are Gaussian, it holds that $\var(||\bx-\bmu_i||^2| \bx \in \Pi_i )=O\{\tr(\bSigma_i^2)\}$ for $i=1,...,k$, so that Condition 3 holds under Condition 2. 
On the other hand, Condition 2 is stronger than Condition 1 since $\lambda_{i1}^2 \le \tr(\bSigma_i^2)$ for $i=1,...,k$. 
\end{remark}

We define $r_j=(-1)^{i+1}(1-\eta_i)$ according to $\bx_j \in \Pi_i$ for $j=1,...,n$. 
The following result gives a geometric representation for (\ref{1.1}) when $k=2$.
\begin{theorem}
\label{thm1}
Assume $\Delta_{1,2} /\tr(\bSigma)\to c \ (> 0)$ as $d\to \infty$. 
Under Conditions 2 to 4, it holds 
\begin{equation}
\plim_{d\to \infty}\frac{(n-1)\bS_D}{\tr(\bSigma)}=c \br \br^T+(1-\varepsilon_1\varepsilon_2c)\bP_n,
\label{2.2}
\end{equation}
where $\br=(r_1,...,r_n)^T$. 
\end{theorem}

From (\ref{2.2}), the first eigenvector of $\bS_D$ uniquely determines the direction. 
In fact, by noting $||\br ||^2=n\eta_1 \eta_2$, we have the following results for the first eigenvector and PC scores when $k=2$.
By using Corollary \ref{cor1}, one can classify $\bx_j$s into two groups by the sign of $\hat{z}_{1j}$s: 
\begin{corollary}
\label{cor1}
Under Conditions 2 to 4, it holds that for $n_i>0,\ i=1,2$ 
\begin{align*}
\plim_{d\to \infty}\hat{\bu}_{1}=\frac{\br}{\sqrt{ n\eta_1 \eta_2}}  \quad  
\mbox{and} \quad 
 \plim_{d\to \infty}\hat{z}_{1j}= \left\{ \begin{array}{ll}
\sqrt{ \eta_2/\eta_1} & \mbox{ when } \bx_j\in \Pi_1, \\[1mm]
-\sqrt{\eta_1/\eta_2} & \mbox{ when } \bx_j\in \Pi_2 
\end{array} \right. \mbox{ for $j=1,...,n$}.
\end{align*}
\end{corollary}

We considered an easy example such as $\Pi_i:N_d(\bmu_i,\bSigma_i),\ i=1,2$, with $\bmu_1=\bze$, $\bmu_2=\bone_d$, $\bSigma_1=(0.3^{|i-j|^{1/3}})$ and $\bSigma_2=\bB(0.3^{|i-j|^{1/3}})\bB$, 
where $\bB=\mbox{diag}[-\{0.5+1/(d+1)\}^{1/2},\{0.5+2/(d+1)\}^{1/2},...,(-1)^d\{0.5+d/(d+1)\}^{1/2}]$. 
We note that $\Delta_{1,2}=d$ and $\bSigma_1 \neq \bSigma_2$ but $\tr(\bSigma_1)=\tr(\bSigma_2)=d$. 
Then, Conditions 2 to 4 hold. 
We set $n_1=1$ and $n_2=2$.
We took $n=3$ samples as $\bx_1 \in \Pi_1$ and $\bx_2,\bx_3 \in \Pi_2$. 
In Fig. 1, we displayed scatter plots of 20 independent pairs of $\pm \hat{\bu}_{1}$ when (a) $d=5$, (b) $d=50$, (c) $d=500$ and (d) $d=5000$.  
We denoted $\br=(2/3,-1/3,-1/3)^T$ by the solid line and $\bone_n=(1,1,1)^T$ by the dotted line. 
We note that $\hat{\bu}_{1}^T \bone_n=0$ when $\bS_D \neq \bO$. 
We observed that all the plots of $\pm \hat{\bu}_1$ gather on the surface of the orthogonal complement of $\bone_n$. 
Also, the plots appeared close to $\br$ as $d$ increases. 
Thus one can classify $\bx_j$s into two groups by the sign of $\hat{z}_{1j}$s. 
If one cannot assume Condition 3 or 4, we recommend to estimate PC scores by using the cross-data-matrix methodology given by \citet{Yata:2010}. 
See \citet{Yata:2010,Yata:2013} for the details. 

\begin{figure}
\centering
\includegraphics[scale=0.385]{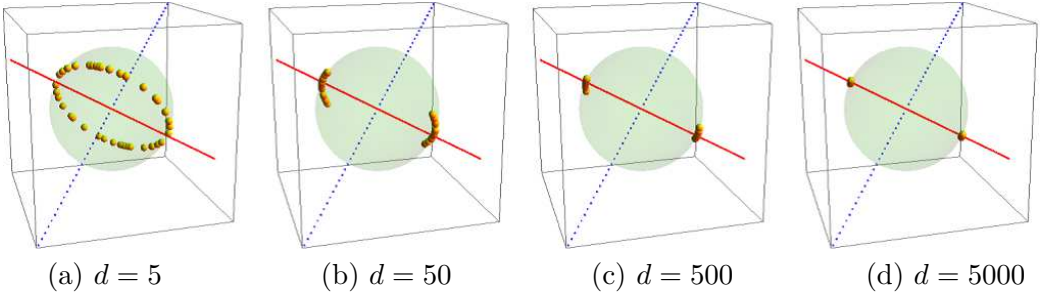}
\\[0.01mm]
\quad (a) $d=5$ \ \ \hspace{1.6cm} (b) $d=50$ \hspace{1.6cm} (c) $d=500$ \hspace{1.5cm} (d) $d=5000$
\caption{\label{fig01}Toy example to illustrate the geometric representation of $\pm \hat{\bu}_1$ on the unit sphere when $k=2$ and $n=3$. 
We plotted 20 independent pairs of $\pm \hat{\bu}_1$ when $\bx_1 \in \Pi_1$ and $\bx_2,\bx_3 \in \Pi_2$.
The solid line denotes $\br=(2/3,-1/3,-1/3)^T$ and the dotted line denotes $\bone_n=(1,1,1)^T$.}
\end{figure}

\section{PC scores for multiclass mixture model}
\subsection{Asymptotic behaviors of true PC scores when $k \ge 3$}
We consider PC scores for the $k\ (\ge 3)$-class mixture model. 
Let $\varepsilon_{(0)}=0$ and $\varepsilon_{(i)}=\sum_{j=1}^i\varepsilon_{j}$ for $i=1,...,k$. 
We assume the condition: 
\begin{condition}
\label{con5}
$\displaystyle \mbox{Angle}(\bmu_{i,i+1},\bmu_{j,j+1})\to \frac{\pi}{2}$ \ and \ $\displaystyle \frac{\Delta_{j,j+1}}{\Delta_{i,i+1}}\to 0$ \ as \ $d\to \infty$ \ for $i,j=1,...,k-1\ (i<j)$.
\end{condition}
We note that $\Delta_{k-1,k}/\Delta_{\min}\to 1$ as $d\to \infty$ under Condition 5. 
Then, we have the following results.
\begin{theorem}
\label{thm2}
Under Conditions 1 and 5, it holds that for $i=1,...,k-1;\ j=1,...,n$ 
\begin{align}
\plim_{d\to \infty} \frac{s_{ij} }{ \lambda_i^{1/2}}= \left\{ \begin{array}{ll}
0 & \mbox{ when $i \ge 2$ and } \bx_j\in \bigcup_{m=1}^{i-1}\Pi_m, \\[1mm]
\sqrt{ (1-\varepsilon_{(i)})/\{\varepsilon_i(1-\varepsilon_{(i-1)})\}} & \mbox{ when } \bx_j \in  \Pi_i, \\[1mm]
-\sqrt{ \varepsilon_{i}/\{(1-\varepsilon_{(i)})(1-\varepsilon_{(i-1)})\}} & \mbox{ when } \bx_j\in \bigcup_{m=i+1}^{k}\Pi_m. 
\end{array} \right. \label{3.1}
\end{align}
\end{theorem}
\begin{remark}
(\ref{1.3}) is equivalent to (\ref{3.1}) with $k=2$ and $i=1$. 
\end{remark}
\begin{corollary}
\label{cor2}
Under Conditions 1 and 5, it holds that for $i=1,...,k-1$
\begin{align*}
\frac{\lambda_i}{\varepsilon_{(i)}(1-\varepsilon_{(i)})\Delta_{i,i+1}/(1-\varepsilon_{(i-1)})}\to 1 \quad \mbox{and}\quad \mbox{Angle}(\bh_i,\bmu_{i,i+1})\to 0\quad \mbox{as $d\to \infty$}. \notag
\end{align*}
\end{corollary}

For example, when $k=3$, from (\ref{3.1}) we have that for $j=1,...,n$
\begin{align*}
&\plim_{d\to \infty} \frac{s_{1j} }{\lambda_1^{1/2}}= \left\{ \begin{array}{ll}
\sqrt{(1-\varepsilon_1)/\varepsilon_1} & \mbox{ when } \bx_j\in \Pi_1, \\[1mm]
-\sqrt{\varepsilon_1/(1-\varepsilon_1)} & \mbox{ when } \bx_j\notin \Pi_1 
\end{array} \right. \\
\mbox{and}\quad 
&\plim_{d\to \infty} \frac{s_{2j} }{\lambda_2^{1/2}}=
\left\{ \begin{array}{ll}
0 & \mbox{ when } \bx_j\in \Pi_1, 
\\[1mm]
\sqrt{ \varepsilon_3/\{\varepsilon_2(1-\varepsilon_1)\}} & \mbox{ when } \bx_j \in  \Pi_2, \\[1mm]
-\sqrt{\varepsilon_2/\{\varepsilon_3(1-\varepsilon_1)\}} & \mbox{ when } \bx_j \in  \Pi_3. 
\end{array} \right.
\end{align*}
One can check whether $\bx_j \in \Pi_1$ or not by the first PC score.
If $\bx_j \notin \Pi_1$, one can check whether $\bx_j \in \Pi_2$ or $\bx_j \in \Pi_3$ by the second PC score. 
In general, one can classify $\bx_j$s by using at most the first $k-1$ PC scores. 

We considered a toy example such as $\Pi_i:N_d(\bmu_i,\bSigma_i),\ i=1,...,4$, where
$\bmu_1=\bone_d$, 
$\bmu_2=(1,...,1,0,...,0)^T$ whose first $\lceil d^{3/4} \rceil $ elements are $1$, 
$\bmu_3=(1,...,1,0,...,0)^T$ whose first $\lceil d^{1/2} \rceil $ elements are $1$, 
and $\bmu_4=\bze$.
Here, $\lceil \cdot \rceil$ denotes the ceiling function.
We set $\bSigma_1=(0.3^{|i-j|^{1/3}})$, $\bSigma_2=\bB(0.3^{|i-j|^{1/3}})\bB$, $\bSigma_3=0.8\bSigma_1$ and $\bSigma_4=1.2 \bSigma_2$, where $\bB$ is defined in Section 2.2. 
Then, Conditions 1 and 5 hold. 
We first considered the case when $k=3:\ \Pi_i,\ i=1,2,3$, having $(\varepsilon_1,\varepsilon_2,\varepsilon_3)=(1/2,1/4,1/4)$. 
We set $n=20$ and $(n_1,n_2,n_3)=(10,5,5)$. 
From Theorem~\ref{thm2} one can expect that $(z_{1j},z_{2j})\ (=(s_{1j}/\lambda_{1}^{1/2},s_{2j}/\lambda_{2}^{1/2}) )$ becomes close to $(1,0)$ when $\bx_j\in \Pi_1$, $(-1,2^{1/2})$ when $\bx_j\in \Pi_2$, and $(-1,-2^{1/2})$ when $\bx_j\in \Pi_3$. 
In Fig. 2, we displayed scatter plots of $(z_{1j},z_{2j})$, $j=1,...,n$, 
when (a) $d=100$, (b) $d=1000$ and (c) $d=10000$. 
We observed that the scatter plots appear close to those three vertices as $d$ increases. 
\begin{figure}
\centering
\includegraphics[scale=0.34]{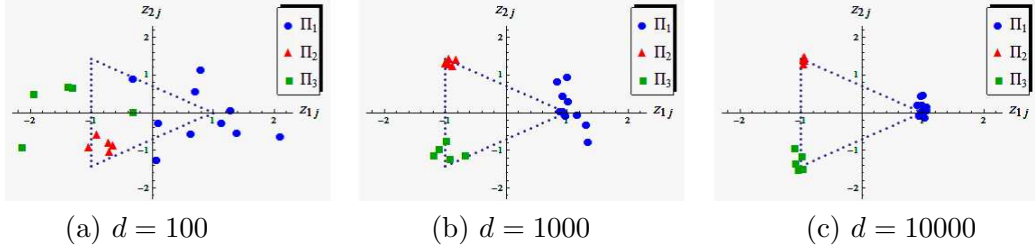}
\\[0.01mm]
(a) $d=100$ \hspace{2.7cm} (b) $d=1000$ \hspace{2.6cm} (c) $d=10000$
\caption{Toy example to illustrate the asymptotic behaviors of true PC scores when $k=3$. 
We plotted $(z_{1j},z_{2j})$ 
which is denoted by small circles when $\bx_j\in \Pi_1$, by small triangles when $\bx_j\in \Pi_2$, and by small squares when $\bx_j\in \Pi_3$. 
The dashed triangle consists of three vertices, $(1,0)$, $(-1,2^{1/2})$ and $(-1,-2^{1/2})$, which are theoretical convergent points.
}
\end{figure}

Next, we considered the case when $k=4:\ \Pi_i,\ i=1,...,4$, having $\varepsilon_1=\cdots =\varepsilon_4=1/4$. 
We set $n=20$ and $n_1=\cdots = n_4=5$.
In Fig. 3, we displayed scatter plots of $(z_{1j},z_{2j},z_{3j})$, $j=1,...,n$, 
when (a) $d=100$, (b) $d=1000$ and (c) $d=10000$. 
From Theorem~\ref{thm2}, we displayed the triangular pyramid given by (\ref{3.1}) with $k=4$.
As expected theoretically, we observed that the scatter plots appear close to four vertices of the triangular pyramid as $d$ increases.
They seemed to converge slower in Fig. 3 than in Fig. 2. 
This is probably because the conditions of Theorem~\ref{thm2} become strict as $k$ increases.
\begin{figure}
\centering
\includegraphics[scale=0.5]{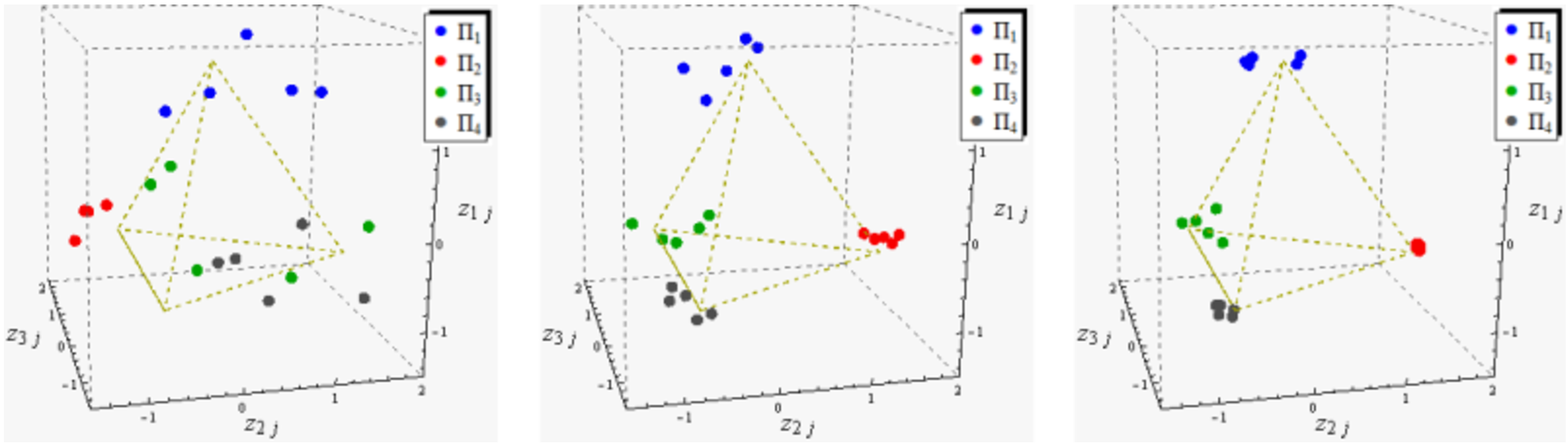}
\\[0.01mm]
(a) $d=100$ \hspace{2.7cm} (b) $d=1000$ \hspace{2.6cm} (c) $d=10000$
\caption{Toy example to illustrate the asymptotic behaviors of true PC scores when $k=4$. 
We plotted $(z_{1j},z_{2j},z_{3j})$. 
The dashed triangular pyramid was given by (\ref{3.1}) with $k=4$.
}
\end{figure}
\subsection{Consistency property of PC scores when $k \ge 3$}
Let $\eta_{(0)}=0$ and $\eta_{(i)}=\sum_{j=1}^i\eta_j$ for $i=1,...,k$. 
We assume the condition:
\begin{condition}
$\displaystyle  \frac{\max_{i=1,...,k-2;j=1,...,k} (\bmu_{i,i+1}^T\bSigma_j\bmu_{i,i+1}) }{\Delta_{\min}^2}\to 0$ \ as \ $d\to \infty$. 
\end{condition}
As for the estimated PC scores, we have the following result.
From Theorem \ref{thm3}, one can classify $\bx_j$s into $k$ groups by the elements of $\hat{\bu}_{i},\ i=1,...,k-1$:
\begin{theorem}
\label{thm3}
Under Conditions 2 to 6, it holds that for $n_i>0$, $i=1,...,k$
\begin{align}
\plim_{d\to \infty}\hat{z}_{ij}= \left\{ 
\begin{array}{ll}
0 & \mbox{ when $i \ge 2$ and } \bx_j\in \bigcup_{m=1}^{i-1}\Pi_m, \\[1mm]
\sqrt{(1-\eta_{(i)})/\{\eta_i(1-\eta_{(i-1)})\}} & \mbox{ when }\bx_j \in  \Pi_i, \\[1mm]
-\sqrt{\eta_{i}/\{(1-\eta_{(i)})(1-\eta_{(i-1)})\}} & \mbox{ when } \bx_j \in \bigcup_{m=i+1}^{k}\Pi_m 
\end{array} \right. 
\label{3.2}
\end{align}
for $i=1,...,k-1;\ j=1,...,n$.
\end{theorem}

\section{Real data examples}
\subsection{Clustering when $k=2$}
We analyzed gene expression data by \citet{Chiaretti:2004} in which the data set consists of $12625\ (=d)$ genes and $128$ samples.
The data set has two tumor cellular subtypes, $\Pi_1:$ B-cell (95 samples) and $\Pi_2:$ T-cell (33 samples). 
Refer to \citet{Jeffery:2006} as well. 
We considered three cases: 
(a) $n=10$ samples consist of the first 5 samples both from $\Pi_1$ and $\Pi_2$ (i.e. $n_1=5$ and $n_2=5$);
(b) $n=40$ samples consist of the first 20 samples both from $\Pi_1$ and $\Pi_2$ (i.e. $n_1=20$ and $n_2=20$); and
(c) $n=128$ samples consist of $n_1=95$ samples from $\Pi_1$ and $n_2=33$ samples from $\Pi_2$. 
In the top panels of Fig. 4, we displayed scatter plots of the first two PC scores, $(\hat{z}_{1j},\hat{z}_{2j})$s, for (a), (b) and (c). 
From Corollary~\ref{cor1}, we denoted $(\eta_2/\eta_1)^{1/2}$ and $-(\eta_1/\eta_2)^{1/2}$ by dotted lines. 
For (a), we observed that the estimated PC scores give good performances. 
The first PC scores gathered around $(\eta_2/\eta_1)^{1/2}$ or $-(\eta_1/\eta_2)^{1/2}$. 
For (b), the estimated PC scores gave adequate performances except for the two points from $\Pi_2$.
Those two samples, which are the ninth and twentieth samples of $\Pi_2$, are probably outliers. 
In fact, the two points are far from the cluster of $\Pi_2$. 
The other $38$ samples were perfectly classified into the two groups by the sign of the first PC scores. 
As for (c), although there seemed to be two clusters except for the two samples, we could not classify the data set by the sign of the first PC scores. 
This is probably because $\eta_1$ and $\eta_2$ are unbalanced and $n$ is large. 
From (\ref{1.2}), when the mixing proportions are unbalanced, $\lambda_1$ becomes small. 
The first eigenspace was possibly affected by the other eigenspaces so that the first PC scores appear in the wrong direction.
We tested the clustering except for the outlying two samples. 
We used the remaining $31$ samples for $\Pi_2$.
We considered three cases for samples from $\Pi_1$: 
(d) the first $16$ samples from $\Pi_1$, so that $n_1=16,\ n_2=31,\ n=47$ and $\eta_1/\eta_2\approx 0.5$; 
(e) the first $31$ samples from $\Pi_1$, so that $n_1=31,\ n_2=31,\ n=62$ and $\eta_1/\eta_2=1$; and
(f) the first $62$ samples from $\Pi_1$, so that $n_1=62,\ n_2=31,\ n=93$ and $\eta_1/\eta_2=2$. 
In the bottom panels of Fig. 4, we displayed scatter plots of $(\hat{z}_{1j},\hat{z}_{2j})$s for (d), (e) and (f).
For (d) and (e), we observed that the estimated PC scores give good performances. 
As for (f), although there seemed to be two clusters, we could not classify the data set by the sign of the first PC scores. 
$\eta_1$ and $\eta_2$ are unbalanced in (d) and (f).
Even though (d) is an unbalanced case, the estimated PC scores worked well for the case. 
We had an estimate of the ratio of the first eigenvalues, $\lambda_{11}/\lambda_{21}$, as $1.598$ by the noise-reduction methodology given by \citet{Yata:2012}. 
The first eigenspace of $\bSigma$ in (d) is less affected by the first eigenspace of $\bSigma_i$s than in (f) since $\bSigma=\varepsilon_1 \varepsilon_{2}\bmu_{1,2}\bmu_{1,2}^T+\varepsilon_1\bSigma_1+\varepsilon_2\bSigma_2$.
This is probably the reason why the estimated PC scores gave good performances even in (d).
\begin{figure}
\centering
\includegraphics[scale=0.39]{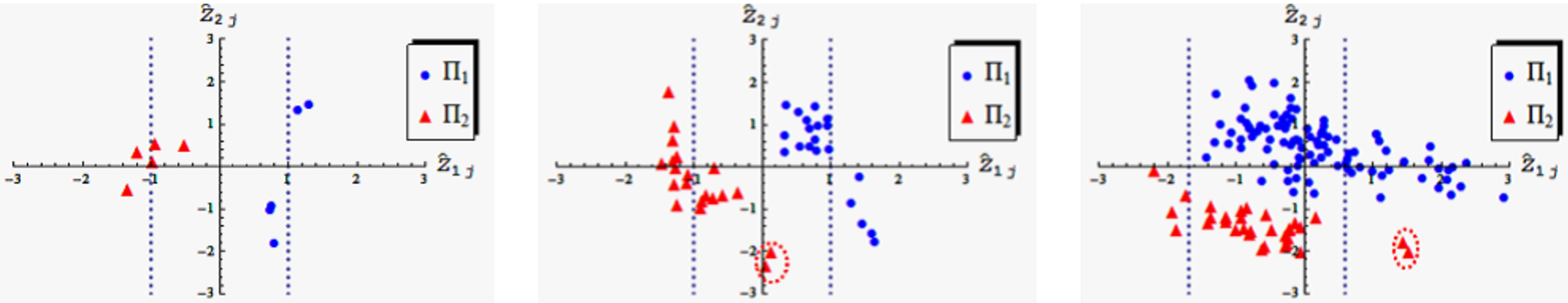}
\\[0.01mm]
(a) $(n_1,n_2)=(5,5)$ \hspace{1.3cm} (b) $(n_1,n_2)=(20,20)$ \hspace{1.2cm} (c) $(n_1,n_2)=(95,33)$ 
\\[2mm]
\includegraphics[scale=0.39]{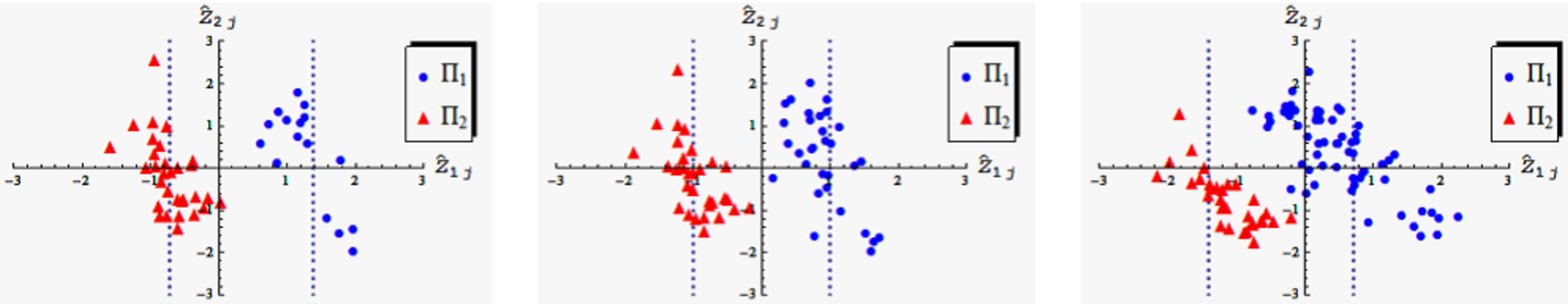}
\\[0.01mm]
(d) $(n_1,n_2)=(16,31)$ \hspace{1.cm} (e) $(n_1,n_2)=(31,31)$ \hspace{1.2cm} (f) $(n_1,n_2)=(62,31)$
\caption{We displayed scatter plots of the first two PC scores, supposing $k=2$ in the data set of \citet{Chiaretti:2004}. 
We denoted them by small circles when $\bx_j\in \Pi_1$ and by small triangles when $\bx_j\in \Pi_2$. 
The theoretical convergent points, $(\eta_2/\eta_1)^{1/2}$ and $-(\eta_1/\eta_2)^{1/2}$, are denoted by dotted lines. 
The two samples, encircled by dots in (b) and (c), are probably outliers. 
}
\end{figure}
\subsection{Clustering when $k\ge 3$}
We analyzed gene expression data by \citet{Pomeroy:2002} in which the data set consists of five brain tumor types.  
However, we only used $4$ classes given in the CRAN R package `rda' in which the data set consists of $5597\ (=d)$ genes and $34$ samples.
We set the four tumor types as $\Pi_1:$ medulloblastomas (10 samples), $\Pi_2:$ malignant gliomas (10 samples), $\Pi_3:$ normal cerebellums (4 samples) and $\Pi_4:$ AT/RT (10 samples). 
We first considered the case when $k=3:\ \Pi_i,\ i=1,2,3$, so that $n_1=10,\ n_2=10,\ n_3=4$ and $n=24$. 
In the left panel of Fig. 5, we displayed scatter plots of the first two PC scores, $(\hat{z}_{1j},\hat{z}_{2j})$s. 
From Theorem~\ref{thm3}, we displayed the triangle given by (\ref{3.2}) with $k=3$.
Although there seemed to be three clusters, we could not observe that they gather around each vertex. 
This is probably because the rate of convergence is slow because of small $d$ compared to such large $n$ when $k\ge 3$. 
We tested the clustering with a small sample size: the first $5$ samples both from $\Pi_1$ and $\Pi_2$ and the last $2$ samples from $\Pi_3$, so that $n_1=5,\ n_2=5,\ n_3=2$ and $n=12$. 
We displayed the results in the right panel of Fig. 5. 
They seemed to be classified into three classes around each vertex.
\begin{figure}[h]
\centering
\includegraphics[scale=0.43]{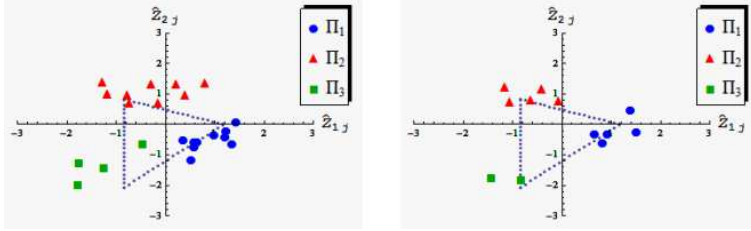}
\\[0.01mm]
(i) $(n_1,n_2,n_3)=(10,10,4)$ \hspace{1.65cm} (ii) $(n_1,n_2,n_3)=(5,5,2)$
\\[0.01mm]
\caption{We displayed scatter plots of the first two PC scores, supposing $k=3$ in the data set of \citet{Pomeroy:2002}. 
We denoted them by small circles when $\bx_j\in \Pi_1$, by small triangles when $\bx_j\in \Pi_2$ and by small squares when $\bx_j\in \Pi_3$. 
The theoretical convergent points are denoted by the vertices of the triangle.}
\end{figure}

Next, we considered the case when $k=4:\ \Pi_i,\ i=1,...,4$, so that $n_1=10,\ n_2=10,\ n_3=4,\ n_4=10$ and $n=34$. 
In Fig. 6, we displayed scatter plots of the first three PC scores. 
Although there seemed to be four clusters of each $\Pi_i$, the data set seemed not to hold the consistency property given by (\ref{3.2}) in Theorem~\ref{thm3}. 
This is probably because some of Conditions 2 to 6 in Theorem~\ref{thm3} are not met because of such large $k$. 
\begin{figure}
\centering
\includegraphics[scale=0.39]{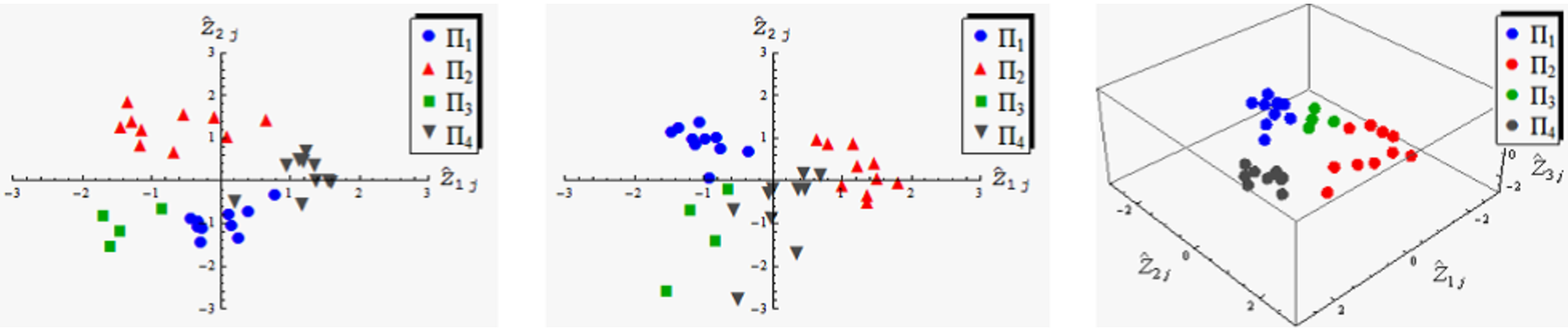}
\\[0.01mm]
(i) $(\hat{z}_{1j},\hat{z}_{2j})$ \hspace{2.2cm} (ii) $(\hat{z}_{2j},\hat{z}_{3j})$ \hspace{2cm} (iii) $(\hat{z}_{1j},\hat{z}_{2j},\hat{z}_{3j})$
\caption{We displayed scatter plots of the first three PC scores, supposing $k=4$ in the data set of \citet{Pomeroy:2002}.}
\end{figure}
\subsection{Clustering: Special case}
We analyzed gene expression data by \citet{Armstrong:2002} in which the data set consists of three leukemia subtypes having $12582\ (=d)$ genes. 
We used $2$ classes such as $\Pi_1$: acute lymphoblastic leukemia ($24$ samples) and $\Pi_2$: mixed-lineage leukemia ($20$ samples), so that $n_1=24,\ n_2=20$ and $n=44$. 
In Fig. 7, we displayed scatter plots of the first three PC scores. 
\begin{figure}[h]
\centering
\includegraphics[scale=0.39]{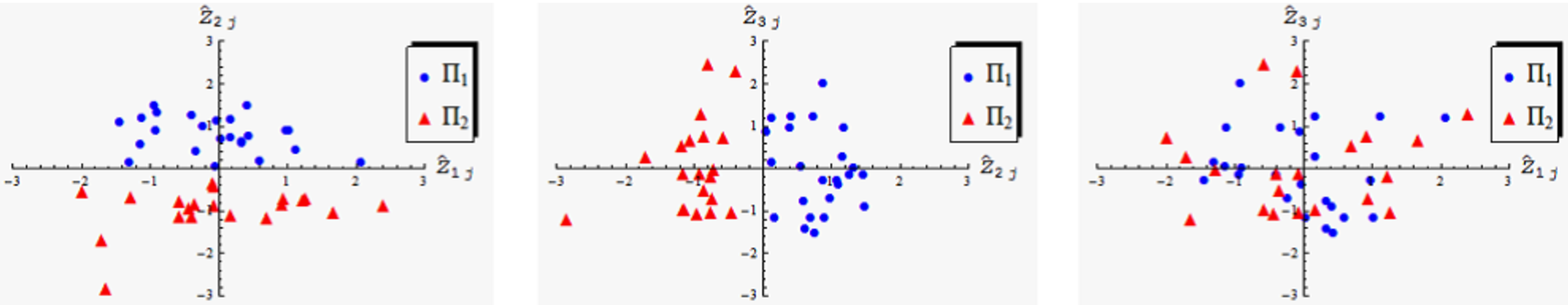}
\\[0.01mm]
(i) $(\hat{z}_{1j},\hat{z}_{2j})$ \hspace{2.2cm} (ii) $(\hat{z}_{2j},\hat{z}_{3j})$ \hspace{2.2cm} (iii) $(\hat{z}_{1j},\hat{z}_{3j})$
\caption{We displayed scatter plots of the first three PC scores, supposing $k=2$ in the data set of \citet{Armstrong:2002}.}
\end{figure}
We observed that the data set is perfectly separated by the sign of the second PC scores. 
This figure looks completely different from Fig. 4. 
This is probably because the largest eigenvalue, $\lambda_{11}$ or $\lambda_{21}$, is too large. 
When $k=2$, we give the following result to explain the reason of the phenomenon in Fig. 7.
Under the assumptions of Proposition~\ref{pro1}, one can classify $\bx_j$s into two groups by some $i$-th PC score even when Condition 1 is not met:
\begin{proposition}
\label{pro1}
Assume $\max_{i=1,2}\bmu_{1,2}^T\bSigma_i\bmu_{1,2}/\Delta_{1,2}^2 \to 0$ as $d\to \infty$. 
Then, there exists some positive integer $i_{\star}$ such that 
$$
\frac{\lambda_{i_{\star}}}{\varepsilon_1\varepsilon_2 \Delta_{1,2}}\to 1\quad \mbox{as $d\to \infty$}.
$$
Furthermore, assume that $\lambda_{i_{\star}}$ is distinct in the sense that $\liminf_{d\to \infty} |{\lambda_{i'}}/{\lambda_{i_{\star}}}-1 |>0$ for $i'=1,...,d\ (i'\neq i_{\star})$.
Then, if $\bh_{i_{\star}}^T\bmu_{1,2} \ge 0$, it holds that Angle$(\bh_{i_{\star}},\bmu_{1,2})\to 0$ as $d\to \infty$ and for $j=1,...,n$
\begin{eqnarray}
\plim_{d\to \infty} \frac{s_{i_{\star}j} }{ \lambda_{i_{\star}}^{1/2} }= \left\{ \begin{array}{ll}
\sqrt{\varepsilon_2/\varepsilon_1} & \mbox{ when } \bx_j\in \Pi_1, \\[1mm]
-\sqrt{\varepsilon_1/\varepsilon_2} & \mbox{ when } \bx_j\in \Pi_2. 
\end{array} \right. \notag 
\end{eqnarray}
\end{proposition}

We estimated the largest eigenvalue by using the noise-reduction methodology given by \citet{Yata:2012}.
We estimated $\Delta_{1,2}$ by using an unbiased estimator given by \citet{Aoshima:2014}. 
Then, we obtained the estimates of $(\lambda_{11}/\Delta_{1,2},\lambda_{21}/\Delta_{1,2})$ as $(0.465,0.787)$, so that Condition 1 is not met obviously. 
In addition, by estimating $\varepsilon_i$s by $\eta_i$s, we had $\varepsilon_{2} \lambda_{21}>\varepsilon_1 \varepsilon_{2} \Delta_{1,2}$. 
Thus, the first eigenspace of $\bSigma$ is probably the first eigenspace of $\bSigma_2$ since $\bSigma=\varepsilon_1 \varepsilon_{2}\bmu_{1,2}\bmu_{1,2}^T+\varepsilon_1\bSigma_1+\varepsilon_2\bSigma_2$. 
We conclude that $i_{\star}$ in Proposition~\ref{pro1} must be $2$.
This is the reason why the data set can be separated by the sign of the second PC scores in Fig 7. 
\section{Concluding remarks}
In this paper, we considered the mixture model by (\ref{1.1}) in the HDLSS context such as $d\to \infty$ while $n$ is fixed. 
We studied asymptotic properties both of the true PC scores and the sample PC scores for the mixture model. 
We gave theoretical reasons why PCA is effective for clustering HDLSS data and we showed that HDLSS data can be classified by the sign of the first several PC scores theoretically. 
However, we have to say, in actual HDLSS data analyses, one may encounter cases such as in Figs. 4(c) and 7 where the data set is not always classified by the sign of the first several PC scores.
Several reasons should be considered:
(i) Actual HDLSS data sets often include several outliers;
(ii) The regularity conditions are not met;
and (iii) $d$ is not sufficiently large. 
Thus, we recommend the following three steps: 
(I) Apply PCA to HDLSS data;
(II) By using PC scores, map the data set onto a feature space such as the first three eigenspaces; 
and (III) Apply general clustering methods such as the $k$-means method to the feature space.

We are now investigating the theory further and hope to bring it closer to the results of actual analysis.
\section*{Acknowledgement}
Research of the first author was partially supported by Grant-in-Aid for Young Scientists (B), Japan Society for the Promotion of Science (JSPS), under Contract Number 26800078.
Research of the second author was partially supported by Grants-in-Aid for Scientific Research (B) and 
Challenging Exploratory Research, JSPS, under Contract Numbers 22300094 and 26540010. 

\appendix
\section{Appendix}
Throughout, let $\bu_i=(u_{i1},...,u_{in})^T$, where 
$$ 
{u}_{ij}= \left\{ 
\begin{array}{ll}
0 & \mbox{ when $i \ge 2$ and } \bx_j\in \bigcup_{m=1}^{i-1}\Pi_m, \\[1mm]
 \sqrt{(1-\eta_{(i)})/\{n\eta_i(1-\eta_{(i-1)})\}} & \mbox{ when } \bx_j \in  \Pi_i, \\[1mm]
-\sqrt{\eta_{i}/\{n(1-\eta_{(i)})(1-\eta_{(i-1)}) \}} & \mbox{ when } \bx_j\in \bigcup_{m=i+1}^{k}\Pi_m 
\end{array} \right. 
$$
for $i=1,...,k-1;\ j=1,...,n$. 
Let $\bnu_i=\sum_{m=1}^k\eta_m(\bmu_i-\bmu_m)$ for $i=1,...,k$. 
Let $\bV=(\bnu_{(1)},...,\bnu_{(n)})$, where $\bnu_{(j)}=\bnu_{i}$ according to $\bx_j\in \Pi_i$ for $j=1,...,n$. 
Note that $\bV \bone_n =\sum_{j=1}^n\bnu_{(j)}=\bze$. 
We define the eigen-decomposition of $\bV^T\bV/n$ by $\bV^T\bV/n=\sum_{i=1}^{k-1}\tilde{\lambda}_{i}\tilde{\bu}_{i}\tilde{\bu}_{i}^T $ from the fact that rank$(\bV)\le k-1$, where $\tilde{\lambda}_{1}\ge\cdots\ge \tilde{\lambda}_{k-1}\ge 0$ are eigenvalues of $\bV^T\bV/n$ and $\tilde{\bu}_{i}=(\tilde{u}_{i1},...,\tilde{u}_{in})^T$ is a unit eigenvector corresponding to $\tilde{\lambda}_{i}$ for each $i$. 
We assume $\tilde{\bu}_i^T{\bu}_i \ge 0$ for $i=1,...,k-1$, without loss of generality.  
\subsection{Lemmas and their proofs}
\begin{lemma}
\label{lem1}
When $k=2$, it holds that under Conditions 2 to 4 
$$
\plim_{d\to \infty}\frac{(n-1)\bS_D-\tr(\bSigma_1)\bP_n}{\Delta_{1,2}}= \br \br^T.
$$ 
\end{lemma}
\begin{proof}
Let $\bmu_{\eta}=\eta_1\bmu_1+\eta_2 \bmu_2 $. 
Then, we can write that $\bx_j-\bmu_{\eta}=(\bx_j-\bmu_i)+(-1)^{i+1} (1-\eta_i)\bmu_{1,2}$ for $j=1,...,n;\ i=1,2$. 
From the fact that $\lambda_{i1} \le \tr(\bSigma_i^2)^{1/2}$, we have that $\var\{(\bx_j-\bmu_i)^T\bmu_{1,2}|\bx_j \in \Pi_i \}=\bmu_{1,2}^T \bSigma_i\bmu_{1,2}\le \Delta_{1,2}\lambda_{i1}=o(\Delta_{1,2}^2)$ as $d\to \infty$ for $j=1,...,n;\ i=1,2$ under Condition 2. 
Also, we have that $\var\{(\bx_j-\bmu_i)^T(\bx_{j'}-\bmu_{i'}) |\bx_j \in \Pi_i, \bx_{j'} \in \Pi_{i'} \}=\tr(\bSigma_i\bSigma_{i'})\le \tr(\bSigma_i^{2})^{1/2}\tr(\bSigma_{i'}^{2})^{1/2}=o(\Delta_{1,2}^2)$ for all $j \neq j'$ and $i,i'=1,2$ under Condition 2. 
Then, by using Chebyshev's inequality, for any $\tau>0$, under Condition 2, it holds that for all $j \neq j'$ and $i,i'=1,2$
\begin{align}
&P \{|(\bx_{j}-\bmu_i)^T (\bx_{j'}-\bmu_{i'})/\Delta_{1,2}| >\tau | \bx_{j}\in \Pi_i,\bx_{j'}\in \Pi_{i'} \}=o(1) \ \mbox{and} \notag \\
&P \{|(\bx_{j}-\bmu_i)^T \bmu_{1,2}/\Delta_{1,2}|>\tau | \bx_{j}\in \Pi_i \}=o(1), 
\label{1}
\end{align}
so that $(\bx_{j}-\bmu_i)^T (\bx_{j'}-\bmu_{i'})/\Delta_{1,2}=o_P(1) $ and $(\bx_{j}-\bmu_i)^T\bmu_{1,2}/\Delta_{1,2}=o_P(1)$ when $\bx_{j}\in \Pi_i$ and $\bx_{j'}\in \Pi_{i'}$ ($j \neq j'$). 
We note that $E(||\bx_{j}-\bmu_i||^2 |\bx_{j}\in \Pi_i)=\tr(\bSigma_i)$. 
Similar to (\ref{1}), under Condition 3, it holds that $\{||\bx_{j}-\bmu_i||^2-\tr(\bSigma_i) \}/\Delta_{1,2}=o_P(1)$ when $\bx_{j}\in \Pi_i$ for $j=1,...,n;$ $i=1,2$.
By noting that $\{\tr(\bSigma_1)-\tr(\bSigma_2)\}/\Delta_{1,2}=o(1)$ under Condition 4, we have that
\begin{equation}
\plim_{d\to \infty} \frac{(\bX-\bmu_{\eta} \bone_n^T)^T(\bX-\bmu_{\eta} \bone_n^T)-\tr(\bSigma_1)\bI_n }{\Delta_{1,2}}= \br \br^T 
\notag
\end{equation}
under Conditions 2 to 4.
By noting that $\bP_n(\bX-\bmu_{\eta} \bone_n^T)^T(\bX-\bmu_{\eta} \bone_n^T) \bP_n/(n-1)=\bS_D$ and $\br^T \bP_n=\br^T$ from $\br^T\bone_{n}=\bze$, we conclude the result.
\end{proof}
\begin{lemma}
\label{lem2}
Let $\acute{\bmu}_{i,i+1}=\bmu_{i,i+1}/\Delta_{i,i+1}^{1/2}$ for $i=1,...,k-1$, and let $\Delta_{(i,j)}=\Delta_{j,j+1}/\Delta_{i,i+1}$ for $i,j=1,...,k-1\ (i<j)$.
Under Conditions 1 and 5, it holds that as $d\to \infty$
\begin{align*}
&\frac{\lambda_i}{\Delta_{i,i+1}}=\frac{\varepsilon_{i}(1-\varepsilon_{(i)})}{1-\varepsilon_{(i-1)}}+o(1)\quad \mbox{and}\quad \bh_i^T\acute{\bmu}_{i,i+1}=1+o(1)\quad \mbox{for  $i=1,...,k-1$};\\
&\bh_{i}^T\acute{\bmu}_{i-1,i}=
-\frac{1-\varepsilon_{(i)}}{1-\varepsilon_{(i-1)}}\Delta_{(i-1,i)}^{1/2}\{1+o(1)\}\quad \mbox{for  $i=2,...,k-1$ when $k\ge 3$};\ and\\
&\bh_{j}^T \acute{\bmu}_{i,i+1}=o(\Delta_{(i,j)}^{1/2}) \quad \mbox{for  $i,j=1,...,k-1$ $(i+1<j)$ when $k\ge 3$}.
\end{align*}
\end{lemma}
\begin{proof}
Let $\be_{d}\ (\in \mathbb{R}^d)$ be an arbitrary unit vector. 
Since $\bSigma= \sum_{i=1}^{k-1} \sum_{j=i+1}^{k} \varepsilon_i \varepsilon_{j}\bmu_{i,j}\bmu_{i,j}^T$
$+\sum_{i=1}^k \varepsilon_i\bSigma_i$, it holds that as $d\to \infty$
\begin{equation}
\frac{\be_{d}^T \bSigma \be_{d}}{\Delta_{k-1,k}}=\frac{\be_{d}^T (\sum_{i=1}^{k-1} \sum_{j=i+1}^{k} \varepsilon_i \varepsilon_{j}\bmu_{i,j}\bmu_{i,j}^T)\be_{d}}{\Delta_{k-1,k}}+o(1)
\label{2}
\end{equation}
under Condition 1.
Note that $\bmu_{i,j}=\sum_{m=i}^{j-1} \bmu_{m,m+1}$ for $i,j=1,...,k\ (i<j)$. 
Thus it holds that 
\begin{align}
&\sum_{i=1}^{k-1} \sum_{j=i+1}^{k} \varepsilon_i \varepsilon_{j}\bmu_{i,j}\bmu_{i,j}^T \notag \\
&=\sum_{i=1}^{k-1}\varepsilon_{(i)}(1-\varepsilon_{(i)}) \bmu_{i,i+1}\bmu_{i,i+1}^T
+\sum_{i=1}^{k-2}\sum_{j=i+1}^{k-1} \varepsilon_{(i)}(1-\varepsilon_{(j)})(\bmu_{i,i+1}\bmu_{j,j+1}^T+\bmu_{j,j+1}\bmu_{i,i+1}^T).
\label{3}
\end{align}
From the fact that $\lambda_1=\bh_1^T\bSigma \bh_1=\max_{\be_d}(\be_d^T \bSigma \be_d)$, by combining (\ref{2}) with (\ref{3}), under Conditions 1 and 5, we have that 
$$
\frac{\lambda_1}{\Delta_{1,2}}=\max_{\be_d}\big\{\varepsilon_{(1)}(1-\varepsilon_{(1)})(\be_d^T \acute{\bmu}_{1,2})^2+o(1)\big\}=
\varepsilon_{(1)}(1-\varepsilon_{(1)})+o(1).
$$
Hence, from the assumption that $\bh_1^T\bmu_{1,2}\ge 0$, it holds that 
$\bh_1^T\acute{\bmu}_{1,2}=1+o(1)$. 

Next, we consider $\lambda_2$ and $\bh_2$. 
Note that $\acute{\bmu}_{i,i+1}^T\acute{\bmu}_{j,j+1}=o(1)$ and $\Delta_{(i,j)}=o(1)$ for $i,j=1,...,k-1\ (i<j)$ under Condition 5. 
Then, under Conditions 1 and 5, it holds that for $j\ge 2$
\begin{align}
0=\frac{\bh_1^T\bSigma \bh_j }{\Delta_{1,2}}=&\varepsilon_{(1)}(1-\varepsilon_{(1)})\{1+o(1)\}\acute{\bmu}_{1,2}^T\bh_j+\varepsilon_{(1)}(1-\varepsilon_{(2)}) \acute{\bmu}_{2,3}^T\bh_j\Delta_{(1,2)}^{1/2}
+o(\Delta_{(1,2)}^{1/2}) \notag
\end{align}
from (\ref{2})-(\ref{3}) and $\bh_1^T\acute{\bmu}_{2,3}=o(1)$, so that for $j\ge 2$
\begin{equation}
\bh_j^T\acute{\bmu}_{1,2}=-\{(1-\varepsilon_{(2)})/(1-\varepsilon_{(1)})\} \acute{\bmu}_{2,3}^T\bh_j\Delta_{(1,2)}^{1/2}+o(\Delta_{(1,2)}^{1/2}). 
\label{4}
\end{equation}
By combining (\ref{2}) with (\ref{3}) and (\ref{4}), we have that
\begin{align}
\frac{\lambda_2}{\Delta_{2,3}}&=\frac{\bh_2^T \bSigma \bh_2}{\Delta_{2,3}} \notag \\
&=\frac{\bh_2^T\{\sum_{i=1}^{2}\varepsilon_{(i)}(1-\varepsilon_{(i)})\bmu_{i,i+1}\bmu_{i,i+1}^T
+\varepsilon_{(1)}(1-\varepsilon_{(2)})(\bmu_{1,2}\bmu_{2,3}^T+\bmu_{2,3}\bmu_{1,2}^T)\}\bh_2}{\Delta_{2,3}}+o(1) \notag \\
&=\varepsilon_{(2)}(1-\varepsilon_{(2)})(\acute{\bmu}_{2,3}^T\bh_2)^2+
\varepsilon_{(1)}(1-\varepsilon_{(1)})\frac{(\acute{\bmu}_{1,2}^T\bh_2)^2}{\Delta_{(1,2)}}+2\varepsilon_{(1)}(1-\varepsilon_{(2)})\frac{(\acute{\bmu}_{1,2}^T\bh_2)(\acute{\bmu}_{2,3}^T\bh_2)}{\Delta_{(1,2)}^{1/2}}\notag \\
&\quad  +o(1) \notag \\
&=\varepsilon_{(2)}(1-\varepsilon_{(2)})-\frac{\varepsilon_{(1)}(1-\varepsilon_{(2)})^2}{1-\varepsilon_{(1)}}+o(1)
=\frac{\varepsilon_{2}(1-\varepsilon_{(2)})}{(1-\varepsilon_{(1)})}+o(1) \label{5}
\end{align}
under Conditions 1 and 5. 
Hence, from the assumption that $\bh_2^T\bmu_{2,3}\ge 0$, it holds that 
$\bh_2^T\acute{\bmu}_{2,3}=1+o(1)$. 

Next, we consider $\lambda_3$ and $\bh_3$. 
Note that $\bh_j^T\acute{\bmu}_{2,3}=o(1)$ for $j\ge 3$ from $\bh_2^T\acute{\bmu}_{2,3}=1+o(1)$. 
Then, under Conditions 1 and 5, we have that for $j\ge 3$
\begin{align}
0=\frac{\bh_1^T\bSigma \bh_j }{\Delta_{1,2}}=&\varepsilon_{(1)}(1-\varepsilon_{(1)})\{1+o(1)\}\acute{\bmu}_{1,2}^T \bh_j+\varepsilon_{(1)}(1-\varepsilon_{(2)}) \{1+o(1)\} \acute{\bmu}_{2,3}^T \bh_j\Delta_{(1,2)}^{1/2} \notag \\
&+\varepsilon_{(1)}(1-\varepsilon_{(3)}) \acute{\bmu}_{3,4}^T \bh_j\Delta_{(1,3)}^{1/2}+o(\Delta_{(1,3)}^{1/2}) \quad \mbox{and} \label{6} \\
0=\frac{\bh_2^T\bSigma \bh_j}{\Delta_{2,3}}=&\varepsilon_{(1)}(1-\varepsilon_{(1)})\frac{\bh_2^T\acute{\bmu}_{1,2} \acute{\bmu}_{1,2}^T\bh_j}{\Delta_{(1,2)}}
+\varepsilon_{(1)}(1-\varepsilon_{(2)})\frac{\bh_2^T(\acute{\bmu}_{1,2}\acute{\bmu}_{2,3}^T+\acute{\bmu}_{2,3}\acute{\bmu}_{1,2}^T)\bh_j}{\Delta_{(1,2)}^{1/2}} \notag \\
&+\varepsilon_{(1)}(1-\varepsilon_{(3)})\frac{\bh_2^T\acute{\bmu}_{1,2}\acute{\bmu}_{3,4}^T\bh_j}{\Delta_{(1,2)}^{1/2}}\Delta_{(2,3)}^{1/2}
+\varepsilon_{(2)}(1-\varepsilon_{(2)})\{1+o(1)\}\acute{\bmu}_{2,3}^T\bh_j\notag \\
&+\varepsilon_{(2)}(1-\varepsilon_{(3)}) \acute{\bmu}_{3,4}^T \bh_j\Delta_{(2,3)}^{1/2}
+o(\Delta_{(2,3)}^{1/2}) \notag \\
=&\frac{\varepsilon_{2}(1-\varepsilon_{(2)})}{1-\varepsilon_{(1)}}\{1+o(1)\}\acute{\bmu}_{2,3}^T \bh_j+
\frac{\varepsilon_{2}(1-\varepsilon_{(3)})}{1-\varepsilon_{(1)}}\acute{\bmu}_{3,4}^T \bh_j\Delta_{(2,3)}^{1/2}
+o(\Delta_{(2,3)}^{1/2}) \notag \\
&+\acute{\bmu}_{1,2}^T \bh_j\times o(\Delta_{(1,2)}^{-1/2}) 
\label{7}
\end{align}
from (\ref{2})-(\ref{4}), $\bh_1^T\acute{\bmu}_{2,3}=o(1)$, $\bh_1^T\acute{\bmu}_{3,4}=o(1)$ and $\bh_2^T\acute{\bmu}_{3,4}=o(1)$.
Then, by combining (\ref{6}) and (\ref{7}), under Conditions 1 and 5, it holds that for $j\ge 3$
\begin{align}
\bh_j^T\acute{\bmu}_{1,2}=o(\Delta_{(1,3)}^{1/2})\quad \mbox{and} \quad 
\bh_j^T\acute{\bmu}_{2,3}=-\{(1-\varepsilon_{(3)})/(1-\varepsilon_{(2)})\} \acute{\bmu}_{3,4}^T\bh_j\Delta_{(2,3)}^{1/2}+o(\Delta_{(2,3)}^{1/2}) 
\label{8}. 
\end{align}
Similar to (\ref{5}), by combining (\ref{2}) with (\ref{3}) and (\ref{8}), under Conditions 1 and 5, we have that 
\begin{align}
\frac{\lambda_3}{\Delta_{3,4}}
&=\varepsilon_{(3)}(1-\varepsilon_{(3)})(\acute{\bmu}_{3,4}^T\bh_3)^2+
\varepsilon_{(2)}(1-\varepsilon_{(2)})\frac{(\acute{\bmu}_{2,3}^T\bh_3)^2}{\Delta_{(2,3)}}+2\varepsilon_{(2)}(1-\varepsilon_{(3)})\frac{(\acute{\bmu}_{2,3}^T\bh_3)(\acute{\bmu}_{3,4}^T\bh_3)}{\Delta_{(2,3)}^{1/2}}\notag \\
&\quad  +o(1) \notag \\
&=\varepsilon_{(3)}(1-\varepsilon_{(3)})-\frac{\varepsilon_{(2)}(1-\varepsilon_{(3)})^2}{1-\varepsilon_{(2)}}+o(1)
=\frac{\varepsilon_{3}(1-\varepsilon_{(3)})}{(1-\varepsilon_{(2)})}+o(1), \notag
\end{align}
so that 
$\bh_3^T\acute{\bmu}_{3,4}=1+o(1)$ from the assumption that $\bh_3^T\bmu_{3,4}\ge 0$.

In a way similar to $\lambda_3$ and $\bh_3$, as for $\lambda_i$ and $\bh_i$ $(4\le i \le k-1)$, 
we have that $\lambda_i/\Delta_{i,i+1}=\varepsilon_{i}(1-\varepsilon_{(i)})/(1-\varepsilon_{(i-1)})+o(1)$, $\bh_i^T\acute{\bmu}_{i,i+1}= 1+o(1)$ and 
$\bh_{i}^T\acute{\bmu}_{i-1,i}=-\{(1-\varepsilon_{(i)})/(1-\varepsilon_{(i-1)})\}\Delta_{(i-1,i)}^{1/2}\{1+o(1)\}$
together with $\bh_{j}^T \acute{\bmu}_{i,i+1}=o(\Delta_{(i,j)}^{1/2})$ for $i,j=1,...,k-1$ $(i+1<j)$ under Conditions 1 and 5.
It concludes the results.
\end{proof}
\begin{lemma}
\label{lem3}
Under Conditions 1 and 5, it holds that for $i=1,...,k-1$
\begin{align*}
&\lim_{d\to \infty} \bh_i^T\sum_{m=1}^k\frac{\varepsilon_m(\bmu_{i'}-\bmu_m)}{\lambda_i^{1/2}} =\left\{\begin{array}{ll}
0 & \mbox{when } i \ge 2 \ \mbox{ and } i'<i, \\[1mm]
\sqrt{(1-\varepsilon_{(i)})/\{\varepsilon_{i}(1-\varepsilon_{(i-1)})\}} & \mbox{when } i' = i, \\[1mm]
-\sqrt{\varepsilon_{i}/\{(1-\varepsilon_{(i)})(1-\varepsilon_{(i-1)})\}} & \mbox{when } i' > i. 
\end{array}\right.
\end{align*} 
\end{lemma}
\begin{proof}
We write that 
\begin{align}
\sum_{m=1}^k\varepsilon_m(\bmu_1-\bmu_m)=\sum_{m=1}^{k-1}(1-\varepsilon_{(m)})\bmu_{m,m+1},\quad \sum_{m=1}^k\varepsilon_m(\bmu_k-\bmu_m)=-\sum_{m=1}^{k-1}\varepsilon_{(m)}\bmu_{m,m+1}\notag \\
\mbox{and}\quad \sum_{m=1}^k\varepsilon_m(\bmu_i-\bmu_m)=\sum_{m=i}^{k-1}(1-\varepsilon_{(m)})\bmu_{m,m+1}-\sum_{m=1}^{i-1}\varepsilon_{(m)}\bmu_{m,m+1}\quad \mbox{for $i=2,...,k-1$}. \label{9} 
\end{align}
By using Lemma~\ref{lem2}, under Conditions 1 and 5, we have that as $d\to \infty$
\begin{align*}
\bh_1^T\sum_{m=1}^k\frac{\varepsilon_m(\bmu_1-\bmu_m)}{\Delta_{1,2}^{1/2}}&=\bh_1^T\frac{(1-\varepsilon_{(1)})\bmu_{1,2}}{\Delta_{1,2}^{1/2}}+o(1)
=1-\varepsilon_{(1)}+o(1) \quad \mbox{and} \\
\bh_1^T\sum_{m=1}^k\frac{\varepsilon_m(\bmu_{i'}-\bmu_m)}{\Delta_{i,i+1}^{1/2}}&=
-\bh_1^T\frac{\varepsilon_{(1)}\bmu_{1,2}}{\Delta_{1,2}^{1/2}}+o(1)
=-\varepsilon_{(1)}+o(1)\quad \mbox{for $i'=2,...,k$}
\end{align*}
from (\ref{9}).
Also, by using Lemma~\ref{lem2}, under Conditions 1 and 5, we have that for $i=2,...,k-1;\ i'=i+1,...,k;\ i''=1,...,i-1$
\begin{align*}
\bh_i^T\sum_{m=1}^k\frac{\varepsilon_m(\bmu_i-\bmu_m)}{\Delta_{i,i+1}^{1/2}}&=\bh_i^T\frac{(1-\varepsilon_{(i)})\bmu_{i,i+1}-\varepsilon_{(i-1)}\bmu_{i-1,i}}{\Delta_{i,i+1}^{1/2}}+o(1)\\
&=(1-\varepsilon_{(i)})+\frac{\varepsilon_{(i-1)}(1-\varepsilon_{(i)})}{1-\varepsilon_{(i-1)}}+o(1)=\frac{1-\varepsilon_{(i)}}{1-\varepsilon_{(i-1)}}+o(1),\\
\bh_i^T\sum_{m=1}^k\frac{\varepsilon_m(\bmu_{i'}-\bmu_m)}{\Delta_{i,i+1}^{1/2}}&=-\varepsilon_{(i)}+\frac{\varepsilon_{(i-1)}(1-\varepsilon_{(i)})}{1-\varepsilon_{(i-1)}}+o(1)
=-\frac{\varepsilon_{i}}{1-\varepsilon_{(i-1)}}+o(1) \\
\mbox{and}\quad \bh_i^T\sum_{m=1}^k\frac{\varepsilon_m(\bmu_{i''}-\bmu_m)}{\Delta_{i,i+1}^{1/2}}&=o(1). 
\end{align*}
Thus, from Lemma~\ref{lem2}, we can conclude the results. 
\end{proof}

\begin{lemma}
\label{lem4}
Assume Conditions 2 to 6.
Then, under the condition: 
\begin{equation}
\plim_{d\to \infty} \frac{\tilde{\lambda}_{i}}{\Delta_{i,i+1}}=c_i\in (0,\infty)\quad \mbox{for $i=1,...,k-1$},
\label{10}
\end{equation}
it holds that 
$$
\plim_{d\to \infty} \hat{\bu}_i^T \tilde{\bu}_i=1\quad \mbox{for $\hat{\bu}_i^T \tilde{\bu}_i \ge 0,\ i=1,...,k-1$}.
$$ 
\end{lemma}
\begin{proof}
We have that $\var\{\bmu_{i,i+1}^T(\bx_j-\bmu_{i'})| \bx_j \in \Pi_{i'}\}=\bmu_{i,i+1}^T \bSigma_{i'} \bmu_{i,i+1}=o(\Delta_{k-1,k}^2)$ as $d\to \infty$ for $j=1,...,n;$ $i=1,...,k-2;$ $i'=1,...,k$, under Condition 6. 
Also, from the fact that $\lambda_{i1}\le  \tr(\bSigma_{i}^2)^{1/2}$, we have that $\var\{\bmu_{k-1,k}^T(\bx_j-\bmu_{i})| \bx_j \in \Pi_{i}\}=\bmu_{k-1,k}^T \bSigma_{i}\bmu_{k-1,k}\le \lambda_{i1} \Delta_{k-1,k}=o(\Delta_{k-1,k}^2)$ for $j=1,...,n;$ $i=1,...,k,$ under Condition 2.
Then, similar to (\ref{1}), under Conditions 2 and 6, it holds that $\bmu_{i,i+1}^T(\bx_j-\bmu_{i'})/\Delta_{k-1,k}=o_P(1)$ when $\bx_j \in \Pi_{i'}$ for $j=1,...,n;\ i=1,...,k-1;\ i'=1,...,k$. 
In addition, under Conditions 2 and 3, we can claim that $(\bx_j-\bmu_{i})^T(\bx_{j'}-\bmu_{i'})/\Delta_{k-1,k}=o_P(1)$ and $||\bx_j-\bmu_{i}||^2/\Delta_{k-1,k}=\tr(\bSigma_i)/\Delta_{k-1,k}+o_P(1)$ when $\bx_j \in \Pi_{i}$ and $\bx_{j'} \in \Pi_{i'}$ for all $j\neq j'$ and $i,i'=1,...,k$. 
Here, we write that $\bx_j-\bmu_{\eta}=(\bx_j-\bmu_i)+\bnu_{i}$ for $j=1,...,n$; $i=1,...,k$, where $\bmu_{\eta}=\sum_{i=1}^k\eta_i\bmu_i$.
Then, by noting (\ref{9}) with $\varepsilon_i=\eta_{i}$ and $\varepsilon_{(i)}=\eta_{(i)}$, $i=1,...,k$, under Conditions 2, 3 and 6, we have that 
\begin{align*}
\frac{||\bx_j-\bmu_{\eta}||^2}{\Delta_{k-1,k}}= \frac{||\bnu_{i}||^2+\tr(\bSigma_i)}{\Delta_{k-1,k}}+o_P(1)
\quad \mbox{and} \quad 
\frac{(\bx_j-\bmu_{\eta})^T(\bx_{j'}-\bmu_{\eta})}{\Delta_{k-1,k}}=
\frac{\bnu_{i}^T\bnu_{i'}}{\Delta_{k-1,k}}+o_P(1)
\end{align*}
when $\bx_j \in \Pi_{i}$ and $\bx_{j'} \in \Pi_{i'}$ for all $j\neq j'$ and $i,i'=1,...,k$. 
Thus, under Conditions 2, 3, 4 and 6, it holds that 
\begin{equation}
\plim_{d\to \infty}\frac{(\bX-\bmu_{\eta} \bone_n^T)^T(\bX-\bmu_{\eta} \bone_n^T)-\tr(\bSigma_1)\bI_n-\bV^T \bV}{\Delta_{k-1,k}}=\bO.
\label{11} 
\end{equation}
Let $\be_{n*}\ (\in \mathbb{R}^n)$ be an arbitrary random unit vector such that $\be_{n*}^T\bone_n=0$. 
We note that $\bP_n(\bX-\bmu_\eta \bone_n^T)^T(\bX-\bmu_\eta \bone_n^T) \bP_n/(n-1)=\bS_D$. 
Then, by noting $\be_{n*}^T\bP_n=\be_{n*}^T$, under (\ref{10}), Conditions 2, 3, 4 and 6, we have that 
\begin{align}
\be_{n*}^T\frac{(n-1) \bS_D-\tr(\bSigma_1)\bI_n}{\Delta_{k-1,k}}\be_{n*}&=\be_{n*}^T\frac{(\bX-\bmu_\eta \bone_n^T)^T(\bX-\bmu_\eta \bone_n^T)-\tr(\bSigma_1)\bI_n}{\Delta_{k-1,k}}\be_{n*} \notag \\
&=\be_{n*}^T\frac{\bV^T \bV}{\Delta_{k-1,k}}\be_{n*}+o_P(1)=\be_{n*}^T\frac{\sum_{i=1}^{k-1}n\tilde{\lambda}_i\tilde{\bu}_i\tilde{\bu}_i^T}{\Delta_{k-1,k}}\be_{n*}+o_P(1) \notag \\
&=\be_{n*}^T\frac{\sum_{i=1}^{n-1}\{(n-1)\hat{\lambda}_i-\tr(\bSigma_1)\}\hat{\bu}_i\hat{\bu}_i^T  }{\Delta_{k-1,k}}\be_{n*}
\label{12}
\end{align}
from (\ref{11}).
We note that $\tilde{\bu}_{i}^T \bone_n=0$ for $i=1,...,k-1$ in case of rank$(\bV)=k-1$. 
Also, we note that $\tilde{\lambda}_{i},\ i=1,...,k-1$, are distinct under Condition 5 and (\ref{10}) for a sufficiently large $d$. 
Thus, if $\hat{\bu}_i^T \tilde{\bu}_i \ge 0$ for $i=1,...,k-1$, we have that $\hat{\bu}_i^T \tilde{\bu}_i=1+o_P(1)$ for $i=1,...,k-1$.
It concludes the result. 
\end{proof}
\begin{lemma}
\label{lem5}
Assume Condition 5. 
For $n_i>0,\ i=1,...,k$, it holds that for $i=1,...,k-1$
$$
\plim_{d\to \infty}\frac{\tilde{\lambda}_{i}}{\Delta_{i,i+1}}=\frac{\eta_{i}(1-\eta_{(i)})}{1-\eta_{(i-1)}} \quad 
\mbox{and}\quad \plim_{d\to \infty}\tilde{\bu}_i^T{\bu}_i=1.
$$
\end{lemma}
\begin{proof}
By noting (\ref{9}) with $\varepsilon_i=\eta_{i}$ and $\varepsilon_{(i)}=\eta_{(i)}$, $i=1,...,k$, we can write that 
\begin{align}
\frac{\bV \bV^T}{n}=&\sum_{i=1}^{k-1}\eta_{(i)}(1-\eta_{(i)})\bmu_{i,i+1} \bmu_{i,i+1}^T\notag \\
&+\sum_{i=1}^{k-2}\sum_{j=i+1}^{k-1} \eta_{(i)}(1-\eta_{(j)})(\bmu_{i,i+1}\bmu_{j,j+1}+\bmu_{j,j+1}\bmu_{i,i+1}).
\label{13}
\end{align}
We have the eigen-decomposition of $\bV \bV^T/n$ by $\bV \bV^T/n=\sum_{i=1}^{k-1}\tilde{\lambda}_{i}\tilde{\bh}_{i}\tilde{\bh}_{i}^T$, where $\tilde{\bh}_{i}$ is a unit eigenvector corresponding to $\tilde{\lambda}_{i}$ for each $i$. 
We note that $\eta_i>0,\ i=1,...,k$ for $n_i>0,\ i=1,...,k$. 
Then, by noting Lemmas~\ref{lem2}-\ref{lem3} and the fact that (\ref{13}) is same as (\ref{3}) with $\varepsilon_{(i)}=\eta_{(i)},\ i=1,...,k-1$, under Condition 5, we have that for $i=1,...,k-1$
$$
\plim_{d\to \infty}\frac{\tilde{\lambda}_{i}}{\Delta_{i,i+1}}=\frac{\eta_{i}(1-\eta_{(i)})}{1-\eta_{(i-1)}}  \quad 
\mbox{and}\quad 
\plim_{d\to \infty} \frac{ \tilde{\bh}_i^T\bnu_{(j)}}{\tilde{\lambda}_i^{1/2}}=u_{ij}n^{1/2}
$$
if $\tilde{\bh}_i^T\bmu_{i,i+1}\ge 0$.
We note that $\tilde{u}_{ij}=\tilde{\bh}_i^T\bnu_{(j)}/(n\tilde{\lambda}_i)^{1/2}$ from the fact that $\tilde{\bu}_{i}=\bV^T\tilde{\bh}_{i}/(n\tilde{\lambda}_{i})^{1/2}$ for $i=1,...,k-1$. 
Hence, we can conclude the result. 
\end{proof}
\subsection{Proofs of the theorems, corollaries and proposition}
\subsubsection{Proofs of Theorem~1 and Corollary~1}
We note that $\tr(\bSigma_1)/\tr(\bSigma)\to (1-\varepsilon_1\varepsilon_2c )$ as $d\to \infty$ under Condition 4 and $\Delta_{1,2} /\tr(\bSigma)\to c \ (> 0)$ as $d\to \infty$. 
Then, by using Lemma~\ref{lem1}, we can conclude the result of Theorem~1. 

Next, we consider the proof of Corollary~1. 
From the fact that $\bone_n^T\bS_D\bone_n=0$, it holds that $\hat{\bu}_1^T \bone_n=0$ when $\bS_D\neq \bO$, so that $\bP_n\hat{\bu}_1=\hat{\bu}_1$. 
Also, note that $||\br ||^2=n\eta_1 \eta_2$. 
Then, by using Lemma~\ref{lem1}, under Conditions 2 to 4, it holds that $\hat{\bu}_1^T\{(n-1)\bS_D-\tr(\bSigma_1)\bP_n\}\hat{\bu}_1/\Delta_{1,2}= n\eta_1 \eta_2+o_P(1)$ as $d\to \infty$.
Hence, from (3) and the assumption that $\hat{\bu}_1^T \bz_1 \ge 0$, we have that $\hat{\bu}_1^T \{(n\eta_1 \eta_2)^{-1/2}\br\}=1+o_P(1)$ as $d\to \infty$ for $n_i>0$, $i=1,2$. 
In view of the elements of $\br$, we can conclude the result of Corollary~1. 
\subsubsection{Proofs of Theorem~2 and Corollary~2}
We write that $\bx_j-\bmu=(\bx_j-\bmu_i)+\sum_{m=1}^k\varepsilon_m(\bmu_i-\bmu_m)$ for $j=1,...,n$; $i=1,...,k$. 
We note that $\var\{\be_d^T(\bx_j-\bmu_i)/\Delta_{\min}^{1/2}|\bx_j \in \Pi_i\}=\be_d^T\bSigma_i \be_d/\Delta_{\min}\le \lambda_{i1}/\Delta_{\min}=o(1)$ as $d\to \infty$ under Condition 1 for 
$j=1,...,n;\ i=1,...,k$, where $\be_{d}\ (\in \mathbb{R}^d)$ is an arbitrary unit vector. 
Then, under Condition 1, when $\bx_j \in \Pi_i$, it holds that as $d\to \infty$
$$
\frac{\be_d^T(\bx_j-\bmu)}{\Delta_{\min}^{1/2}}=\frac{\be_d^T\{\sum_{m=1}^k \varepsilon_m(\bmu_i-\bmu_m)\}}{\Delta_{\min}^{1/2}}+o_P(1).
$$
Then, by using Lemmas~\ref{lem2} and \ref{lem3}, we can conclude the result of Theorem~2.

For the proof of Corollary~2, from Lemma~\ref{lem2}, the results are obtained straightforwardly.
\subsubsection{Proof of Theorem~3}
By combining Lemmas~\ref{lem4} and \ref{lem5}, from Theorem 2 and the assumption that $\hat{\bu}_{i}^T\bz_i \ge 0$ for all $i$, the result is obtained straightforwardly. 
\subsubsection{Proof of Proposition 1}
Let $\bSigma_{(*)}=\varepsilon_1 \bSigma_1+\varepsilon_2 \bSigma_2$. 
Then, we define the eigen-decomposition of $\bSigma_{(*)}$ by $\bSigma_{(*)}=\sum_{i=1}^{d}{\lambda}_{i(*)}{\bh}_{i(*)}{\bh}_{i(*)}^T $, where $\lambda_{1(*)}\ge\cdots\ge {\lambda}_{d (*)}\ge 0$ are eigenvalues of $\bSigma_{(*)}$ and ${\bh}_{i(*)}$ is a unit eigenvector corresponding to ${\lambda}_{i(*)}$ for each $i$. 
Let $\lambda=\varepsilon_1 \varepsilon_{2} \Delta_{1,2}$. 
Then, from $\bSigma=\lambda \acute{\bmu}_{1,2} \acute{\bmu}_{1,2}^T+\bSigma_{(*)}$, under $\max_{i=1,2}\acute{\bmu}_{1,2}^T \bSigma_i \acute{\bmu}_{1,2}/\Delta_{1,2} \to 0$ as $d\to \infty$, it holds that $\acute{\bmu}_{1,2}^T \bSigma \acute{\bmu}_{1,2}/\lambda \to 1$ as $d\to \infty$, so that 
\begin{equation}
\sum_{i=1}^d \frac{\lambda_{i(*)} (\bh_{i(*)}^T \acute{\bmu}_{1,2})^2}{\lambda}=o(1),
\label{14}
\end{equation} 
where $\acute{\bmu}_{1,2}={\bmu}_{1,2}/\Delta_{1,2}^{1/2}$. 
Let $\kappa(i)=\lambda_{i(*)}-\lambda$ for $i=1,...,d$. 
For a sufficiently large $d$, when $\kappa(1)>0$, there exists some positive integer $i_*$ such that $i_*=\max \{i|\kappa(i)>0\ \mbox{for $i=1,...,d$}\}$. 
Then, from (\ref{14}), we have that $\sum_{i=1}^{i_*}(\bh_{i(*)}^T \acute{\bmu}_{1,2})^2=o(1)$, so that $\lambda_{i_\star}/\lambda=1+o(1)$ with $i_{\star}=i_*+1$. 
When $\kappa(1)\le 0$ for a sufficiently large $d$, 
it holds that $\lambda_{i_\star}/\lambda=1+o(1)$ with $i_{\star}=1$.
In addition, under $ \liminf_{d\to \infty} |{\lambda_{i'}}/{\lambda_{i_{\star}}}-1 |>0$ for $i'=1,...,d\ (i'\neq i_{\star})$, it holds that $\bh_{i_{\star}}^T \bmu_{1,2}= 1+o(1)$ from $\bh_{i_{\star}}^T \bmu_{1,2}\ge 0$. 
Then, from the fact that $\bh_{i_{\star}}^T\bSigma_{i} \bh_{i_{\star}}/\lambda \to 0$ as $d\to \infty$ for $i=1,2$, in a way similar to (\ref{1}), we have that 
$$
\frac{ s_{i_{\star}j}}{\lambda_{i_{\star}}^{1/2}}=\frac{\bh_{i_{\star}}^T(\bx_j-\bmu)}{\lambda_{i_{\star}}^{1/2} }
=\frac{ \bh_{i_{\star}}^T(\bmu_i-\bmu)}{\lambda_{i_{\star}}^{1/2} }+o_P(1)
$$
when $\bx_j \in \Pi_i$ for $j=1,...,n;\ i=1,2$. 
We can conclude the results.


\begin{thebibliography}{99}

\bibitem[Ahn et~al.(2007)]{Ahn:2007}
Ahn, J., Marron, J.~S., Muller, K.~E. and Chi, Y.~Y. (2007)
\newblock The high-dimension, low-sample-size geometric representation holds under mild conditions. 
\newblock \textit{Biometrika} \textbf{94}, 760--766.

\bibitem[Ahn et~al.(2012)]{Ahn:2012}
Ahn, J., Lee, M.~H. and Yoon, Y.~J. (2012) 
\newblock Clustering high dimension, low sample size data using the maximal data piling distance. 
\newblock \textit{Statist. Sin.}, \textbf{22}, 443--464.

\bibitem[Aoshima and Yata(2014)]{Aoshima:2014}
Aoshima, M. and Yata, K. (2014)
\newblock A distance-based, misclassification rate adjusted classifier for multiclass, high-dimensional data. 
\newblock \textit{Ann. Inst. Statist. Math.}, \textbf{66}, 983--1010.

\bibitem[Armstrong et~al.(2002)]{Armstrong:2002}
Armstrong, S.~A., Staunton, J.~E., Silverman, L.~B., Pieters, R. den~Boer, M.~L., Minden, M.~D., Sallan, S.~E., 
Lander, E.~S., Golub, T.~R. and Korsmeyer, S.~J. (2002)
\newblock MLL translocations specify a distinct gene expression profile that distinguishes a unique leukemia. 
\newblock \textit{Nature Genetics}, \textbf{30}, 41--47. 



\bibitem[Chiaretti et~al.(2004)]{Chiaretti:2004}
Chiaretti, S., Li, X., Gentleman, R., Vitale, A., Vignetti, M., Mandelli, F., Ritz, J. and Foa, R. (2004)
\newblock Gene expression profile of adult T-cell acute lymphocytic leukemia identifies distinct subsets of patients with different response to therapy and survival. 
\newblock \textit{Blood}, \textbf{103}, 2771--2778. 


\bibitem[Hall et~al.(2005)]{Hall:2005}
Hall, P., Marron, J.~S. and Neeman, A. (2005)
\newblock Geometric representation of high dimension, low sample size data.
\newblock \textit{J. R. Statist. Soc. {\rm B}}, \textbf{67}, 427--444.

\bibitem[{Hellton and Thoresen(2014)}]{Hellton:2014}
Hellton, K. and Thoresen, M. (2014). 
\newblock Asymptotic distribution of principal component scores for pervasive, high-dimensional eigenvectors.
\newblock \textit{arXiv preprint arXiv:1401.2781}.

\bibitem[Jeffery et~al.(2006)]{Jeffery:2006}
Jeffery, I.~B., Higgins, D.~G. and Culhane, A.~C. (2006)
\newblock Comparison and evaluation of methods for generating differentially expressed gene lists from microarray data. 
\newblock \textit{BMC Bioinformatics}, \textbf{7}, 359.

\bibitem[Jolliffe(2002)]{Jolliffe:2002}
Jolliffe, I.~T. (2002)
\newblock \textit{{Principal Component Analysis}}.
\newblock New York: Springer.

\bibitem[Johnstone(2001)]{Johnstone:2001}
Johnstone, I.~M. (2001) 
\newblock On the distribution of the largest eigenvalue in principal components analysis. 
\newblock \textit{Ann. Statist.}, \textbf{29}, 295--327. 

\bibitem[Jung and Marron(2009)]{Jung:2009}
Jung, S. and Marron, J.~S. (2009)
\newblock PCA consistency in high dimension, low sample size context. 
\newblock \textit{Ann. Statist.}, \textbf{37}, 4104--4130.


\bibitem[Liu et~al.(2008)]{Liu:2008}
Liu, Y., Hayes, D.~N., Nobel, A. and Marron, J.~S. (2008) 
\newblock Statistical significance of clustering for high-dimension, low-sample size data.
\newblock \textit{J. Am. Statist. Ass.}, \textbf{103}, 1281--1293.

\bibitem[Lv(2013)]{Lv:2013}
Lv, J. (2013)
\newblock Impacts of high dimensionality in finite samples.
\newblock \textit{Ann. Statist.}, \textbf{41}, 2236--2262.

\bibitem[Pomeroy et~al.(2002)]{Pomeroy:2002}
Pomeroy, S.~L., Tamayo, P., Gaasenbeek, M., 
Sturla, L.~M., Angelo, M., McLaughlin, M.~E., 
Kim, J.~Y., Goumnerova, L.~C., Black, P.~M.,
Lau, C. 
et al. (2002)
\newblock Prediction of central nervous system embryonal tumour outcome based on gene expression.
\newblock \textit{Nature}, \textbf{415}, 436--442.


\bibitem[Qiao et~al.(2010)]{Qiao:2010}
Qiao, X., Zhang, H.~H., Liu, Y., Todd, M.~J. and Marron, J.~S. (2010)
\newblock Weighted distance weighted discrimination and its asymptotic properties. 
\newblock \textit{J. Am. Statist. Ass.}, \textbf{105}, 401--414.


\bibitem[Yata and Aoshima(2010)]{Yata:2010}
Yata, K. and Aoshima, M. (2010)
\newblock Effective PCA for high-dimension, low-sample-size data with singular value decomposition of cross data matrix.  
\newblock \textit{J. Multiv. Anal.}, \textbf{101}, 2060--2077.

\bibitem[Yata and Aoshima(2012)]{Yata:2012}
\textsc{Yata, K.} and \textsc{Aoshima, M.} (2012)
\newblock Effective PCA for high-dimension, low-sample-size data with noise reduction via geometric representations.   
\newblock \textit{J. Multiv. Anal.}, \textbf{105}, 193--215.

\bibitem[Yata and Aoshima(2013)]{Yata:2013}
Yata, K. and Aoshima, M. (2013)
\newblock PCA consistency for the power spiked model in high-dimensional settings.   
\newblock \textit{J. Multiv. Anal.}, \textbf{122},  334--354.


\end{thebibliography}
\end{document}